\documentclass[a4paper]{article}

\pdfoutput=1 

\usepackage[T1]{fontenc}
\usepackage[utf8]{inputenc}
\usepackage[english]{babel}
\usepackage{graphics}
\usepackage{epsfig}
\usepackage{amsmath}
\usepackage{amssymb}
\usepackage{amsthm}
\usepackage{url}
\usepackage{hyperref}
\usepackage{float}
\usepackage{pifont}
\usepackage[table]{xcolor}
\usepackage{lmodern}
\usepackage{xspace}
\usepackage{bbold}
\usepackage[capitalize,nameinlink]{cleveref}
\usepackage{subcaption}
\usepackage{mathtools}

\usepackage{draftwatermark} 
\SetWatermarkText{Preprint} 
\SetWatermarkScale{1}       

\usepackage{graphicx}

\usepackage[english]{babel}
\usepackage{amsmath,amsfonts,amssymb}

\usepackage{hyperref}

\usepackage{subcaption}

\newcommand{\R}{\mathbb{R}}

\newcommand{\T}{\intercal}
\newcommand{\e}{e}
\newcommand{\D}{d}

\DeclareMathOperator{\tr}{tr}
\DeclareMathOperator*{\argmax}{arg\,max}

\newtheoremstyle{thm}{\topsep}{\topsep}{\normalfont \itshape}{}{\normalfont \bfseries}{}{\newline}{}
\theoremstyle{thm}

\newtheorem{lemma}{Lemma}
\newtheorem{remark}{Remark}
\newtheorem*{errorin}{Error Indicator}

\definecolor{darkgreen}{RGB}{0,128,0}

\newtagform{fn}{(}{)\footnotemark}

\begin{document}

\title{Cross-Gramian-Based Dominant Subspaces}



\author{Peter~Benner\thanks{Computational Methods in Systems and Control Theory, Max Planck Institute for \newline Dynamics of Complex Technical Systems, Sandtorstr. 1, 39106 Magdeburg, Germany; \newline Faculty of Mathematics, Otto von Guericke University Magdeburg, Universit\"atsplatz 2, 39106 Magdeburg, Germany; ORCID: 0000-0003-3362-4103, \url{benner@mpi-magdeburg.mpg.de}} \and
        Christian~Himpe\thanks{Computational Methods in Systems and Control Theory, Max Planck Institute for \newline Dynamics of Complex Technical Systems, Sandtorstr. 1, 39106 Magdeburg, Germany; \newline ORCID: 0000-0003-2194-6754, \url{himpe@mpi-magdeburg.mpg.de}}
}


\date{}

\maketitle

\begin{abstract}
A standard approach for model reduction of linear input-output systems is balanced truncation,
which is based on the controllability and observability properties of the underlying system.
The related dominant subspaces projection model reduction method similarly utilizes these system properties,
yet instead of balancing, the associated subspaces are directly conjoined.
In this work we extend the dominant subspace approach by computation via the cross Gramian for linear systems,
and describe an a-priori error indicator for this method.
Furthermore, efficient computation is discussed alongside numerical examples illustrating these findings.

~\\


~\\

\end{abstract}




\section{Introduction}
Input-output systems map an input function to an output function via a dynamical system.
The input excites or perturbs the state of the dynamical system and the output is some transformation of the state.
Typically, these input and output functions are low-dimensional while the intermediate dynamical system is high(er)-dimensional.
In applications from natural sciences and engineering,
the dimensionality of the dynamical system may render the numerical computation of outputs from inputs excessively expensive or at least demanding.

Model reduction addresses this computational challenge by algorithms that provide surrogate systems,
which approximate the input-output mapping of the original system with a low(er)-dimensional intermediate dynamical system.
Practically, the trajectory of the dynamical system's state is constrained to a subspace of the original system's state-space,
for example by using truncated projections.

\pagebreak

A standard approach for projection-based model reduction of input-output systems is balanced truncation \cite{morMoo81},
which transforms the state-space unitarily to a representation that is sorted (balanced) in terms of the input's effect on the state (controllability)
as well as the state's effect on the output (observability) and discards (truncates) the least important states according to this measure.

Instead of balancing, this work investigates a dominant subspaces approach \cite{morPen06},
that conjoins the most controllable and most observable subspaces into a projection.
This unbalanced model reduction method may yield larger or less accurate reduced-order systems,
yet allows a computationally advantageous formulation while also preserving stability and providing an error quantification.
The dominant subspace model reduction method has been investigated in \cite{morPen06,morLiW99,morLiW01,morShiS05,morBauBF14},
with \cite{morPen06} being the original source which is already referenced by the earlier work \cite{morLiW01}.

The approach proposed in this work,
combines the method from \cite{morPen06} with the cross Gramian (matrix) \cite{morFerN83},
which encodes controllability \emph{and} observability information of an underlying input-output system.
For this cross-Gramian-based dominant subspace method, an a-priori error indicator is developed,
and the numerical issues arising in the wake of large-scale systems are addressed,
specifically by utilizing the hierarchical approximate proper orthogonal decomposition (HAPOD) \cite{morHimLR18a}.
Compared to other cross Gramian and SVD model reduction techniques such as \cite{morJiaQY19},
the proposed method does not need multiple decompositions, but a single HAPOD.

The considered class of input-output systems are generalized linear \linebreak (time-invariant) systems\footnote{Sometimes, the term descriptor system is used for this type of system,
yet typically descriptor systems explicitly allow a singular mass matrix.
Hence, we decided to use the term \emph{generalized linear system}.},
mapping input $u:\R \to \R^M$ via the state $x:\R \to \R^N$ --- a solution to an ordinary differential equation --- to the output $y:\R \to \R^Q$:
\begin{align}\label{eq:sys}
\begin{split}
 E\dot{x}(t) &= Ax(t) + Bu(t), \\
        y(t) &= Cx(t),
\end{split}
\end{align}
with a system matrix $A \in \R^{N \times N}$, an input matrix $B \in \R^{N \times M}$,
an output matrix $C \in \R^{Q \times N}$ and a mass matrix $E \in \R^{N \times N}$.
In the scope of this work, we assume $E$ to be \emph{non-singular} as well as the matrix pencil $(A,E)$ to be asymptotically stable,
meaning the eigenvalues of the associated generalized eigenproblem lie in the open left half-plane.
This type of system arises, for example, in spatial discretizations of partial differential equations using the finite element method.

In \cref{sec:cg} the cross Gramian for generalized linear systems is introduced,
followed by \cref{sec:mr}, briefly describing projection-based model reduction,
and extending the dominant subspace projection method to the cross Gramian together with an error indicator.
The proposed model reduction technique is then tested numerically in \cref{sec:nr} and a summary is given in \cref{sec:conc}.

\pagebreak

\section{Generalized Cross Gramian}\label{sec:cg}
In this section, the cross Gramian matrix, introduced in \cite{morFerN83},
is briefly reviewed from the point of view of generalized linear time-invariant (LTI) systems \eqref{eq:sys}.

Fundamental to system-theoretic model reduction are the controllability and observability operators \cite{morAnt05},
which are given for \eqref{eq:sys} by the generalized controllability operator $\mathcal{C}: L_2 \to \R^N$ and the generalized observability operator $\mathcal{O}: \R^N \to L_2$:
\begin{align*}
 \mathcal{C}(u) &:= \int_0^\infty \e^{E^{-1} At} E^{-1} Bu(t) \D t, \\
 \mathcal{O}(x_0) &:= C \e^{E^{-1} At} E^{-1} x_0.
\end{align*}
The (generalized) cross Gramian\footnote{Note that the term \emph{generalized cross Gramian} is used in \cite{morSha12} for cross Gramians of unstable systems.}
is then defined as a composition of the generalized controllability and observability operators:
\begin{align}\label{eq:wx}
 W_X := \mathcal{C} \circ \mathcal{O} = \int_0^\infty \e^{E^{-1} At} E^{-1} BC \e^{E^{-1} At} E^{-1} \D t \in \R^{N \times N},
\end{align}
and jointly quantifies controllability and observability of square systems -- systems with the same number of inputs and outputs $M = Q$.
For linear, square systems with $E = I$, the cross Gramian solves a Sylvester matrix equation \cite{morFerN83};
for $E \neq I$, the generalized cross Gramian solves a Sylvester-type equation:
\begin{align*}
 A W_X E + E W_X A = -B C,
\end{align*}
which can be shown using integration-by-parts of \eqref{eq:wx}:
\begin{align*}
 W_X &= \int_0^\infty \e^{E^{-1} At} E^{-1} BC \e^{E^{-1} At} E^{-1} \D t \\
     &= (E^{-1} A)^{-1} \e^{E^{-1} At} E^{-1} BC \e^{E^{-1} At} E^{-1} \Big|_0^\infty \\
     &\quad - (E^{-1} A)^{-1} \int_0^\infty \e^{E^{-1} At} E^{-1} BC \e^{E^{-1} At} (E^{-1} A) E^{-1} \D t \\
 &\Rightarrow A W_X = E \e^{E^{-1} At} E^{-1} BC \e^{E^{-1} At} E^{-1} \Big|_0^\infty - E W_X A E^{-1} \\
 &\Rightarrow A W_X E + E W_X A = E \e^{E^{-1} At} E^{-1} BC \e^{E^{-1} At} \Big|_0^\infty = -BC.
\end{align*}
Besides the cross Gramian, the (generalized) controllability Gramian $W_C := \mathcal{CC^*}$
and (generalized) observability Gramian $W_O := \mathcal{O^* O}$ are defined accordingly \cite{Saa09,morSty04}.
For systems with a symmetric Hankel operator $H := \mathcal{OC}$, $H = H^*$ \cite{morOpmR15},
for example all SISO (Single-Input-Single-Output) systems,
the (generalized) cross Gramian has the property:
\begin{align}\label{eq:hsv}
 W_X W_X = \mathcal{COCO} = \mathcal{C(OC)^*O} = \mathcal{CC^*O^*O} = W_C W_O.
\end{align}

\pagebreak

Hence for symmetric systems, either, $W_X$ or $\{W_C, W_O\}$ can be used interchangeably,
if controllability and observability are to be concurrently evaluated.
For non-symmetric and especially non-square systems, an approximation to the cross Gramian is defined,
based on the column-wise partitioning of the input matrix $B$ and row-wise partitioning of the output matrix $C$:
\begin{align*}
 B = \begin{pmatrix} b_1 & \dots & b_M \end{pmatrix}, \quad C = \begin{pmatrix} c_1 & \dots & c_Q \end{pmatrix}^\T.
\end{align*}
For $\bar{B} := \sum_{m=1}^M b_m$ and $\bar{C} := \sum_{q=1}^Q c_q^\T$,
the non-symmetric generalized cross Gramian \cite{morHimO16} for \eqref{eq:sys} is defined as:
\begin{align}\label{eq:avgsys}
 W_Z := \int_0^\infty \e^{E^{-1} At} E^{-1} \bar{B} \bar{C} \e^{E^{-1} At} E^{-1} \D t,
\end{align}
which is the cross Gramian of the average system $(E, A, \bar{B}, \bar{C})$.

The (non-symmetric) generalized cross Gramian \eqref{eq:wx} can be computed numerically,
for example, using the Hessenberg-Schur algorithm \cite{GarLAetal92},
the alternating direction implicit (ADI) algorithm \cite{BenLT09,BenK14,BenKS14b},
or as an empirical cross Gramian \cite{morHim18b}.\linebreak

\section{Model Reduction}\label{sec:mr}
One of the main numerical applications of the cross Gramian is model (order) reduction,
which aims to determine lower order surrogate systems for \eqref{eq:sys}, with respect to the state-space dimension $N := \dim(x(t))$.
The reduced order model (ROM) with $x_r : \R \to \R^n$, $n \ll N$,
\begin{align*}
 E_r \dot{x}_r(t) &= A_r x_r(t) + B_r u(t), \\
     \tilde{y}(t) &= C_r x_r(t),
\end{align*}
has a reduced system matrix $A_r \in \R^{n \times n}$,
a reduced input matrix $B_r \in \R^{n \times M}$,
a reduced output matrix $C_r \in \R^{Q \times n}$ and a reduced mass matrix $E_r \in \R^{n \times n}$,
such that the reduced system's output $\tilde{y} : \R \to \R^Q$ approximates the full order model's output:
\begin{align*}
 \frac{\|y - \tilde{y}\|}{\|y\|} \ll 1,
\end{align*}
in a suitable norm.

Following, the projection-based dominant subspaces model reduction method is extended to exploit the cross Gramian for computation,
and the practical computation of the cross-Gramian-based dominant subspaces is discussed.

\subsection{Projection-Based Model Reduction}
A commonplace approach to construct reduced order models is mapping the state-space trajectory $x(t)$ to a lower dimensional subspace,
using a reduction operator $V_1 : \R^N \to \R^n$ and a lifting operator $U_1 : \R^n \to \R^N$ \cite{morPer16}:
\begin{align*}
 x_r(t) := V_1 x(t) \quad \rightarrow \quad x(t) \approx U_1 x_r(t).
\end{align*}
In the case of (generalized) linear systems \eqref{eq:sys}, the operators $U_1 \in \R^{N \times n}$ and $V_1 \in \R^{n \times N}$,
can be directly applied to the system components $A$, $B$, $C$ and $E$ to obtain the reduced quantities:
\begin{align}\label{eq:rom}
 A_r := V_1 A U_1, \quad B_r := V_1 B, \quad C_r := C U_1, \quad E_r := V_1 E U_1.
\end{align}
Hence, the aim is the computation of suitable reducing and lifting operators $U_1$, $V_1$,
which are typically assumed to be bi-orthogonal $V_1 U_1 = I$.
The dominant subspaces method, considered in this work, is additionally orthogonal $V_1 := U_1^\T$,
thus, the reduction process is a Galerkin projection,
which is stability preserving, if the symmetric part of the system matrix $A$ is negative definite,
and the mass matrix $E$ positive definite \cite[Sec.~II.C]{morBonD08} (strictly dissipative systems),
\begin{align}\label{eq:stab}
 A + A^\T < 0 \quad \land \quad E > 0.
\end{align}
This is a generalization of the stability preservation for systems with \linebreak $E = I$, mentioned in \cite[Sec.~4.3]{morPen06}.
If a system does not fulfill \eqref{eq:stab}, a stabilization procedure, see for example \cite[Sec.~4]{morBenHM18},
can be applied to the reduced order model.

\subsection{Dominant Subspaces}
The \emph{Dominant Subspaces Projection Model Reduction} (DSPMR) is introduced in \cite[Sec.~4.3]{morPen06}.
The idea behind DSPMR is, instead of balancing controllability and observability Gramians,
to combine the associated principal subspaces obtained from approximate system Gramians.
This yields a simple model reduction algorithm which is based upon low-rank factors of the controllability and observability Gramians.
In \cite{morPen06}, a low-rank Cholesky (LR Chol) factor is used, while \cite{morLiW99} utilizes singular vectors of a truncated singular value decomposition (tSVD),
\begin{align*}
 W_C &\stackrel{\text{LR Chol}}{\approx} Z_C Z_C^\T, &\quad W_O &\stackrel{\text{LR Chol}}{\approx} Z_O Z_O^\T, \\
 W_C &\stackrel{\text{~tSVD~}}{\approx} U_C D_C U_C^\T, &\quad W_O &\stackrel{\text{~tSVD~}}{\approx} U_O D_O U_O^\T.
\end{align*}
The controllability and observability subspaces encoded in the matrix factors are now conjoined and orthogonalized,
by either a rank-revealing SVD (\cite{morPen06}) or a rank-revealing QR-decomposition (\cite{morLiW99,morLiW01}).
Either, the left singular vectors $U$, or the $Q$ factor, can be taken as Galerkin projections, respectively:
\begin{align*}
 QR &\stackrel{\text{QR}}{=} \begin{bmatrix} U_C & U_O\end{bmatrix} \rightarrow U_1 := Q, \\
 UDV^\T &\stackrel{\text{SVD}}{=} \begin{bmatrix}U_C & U_O\end{bmatrix} \rightarrow U_1 := U,
\end{align*}
see also \cite[Sec~2.1.7]{morBauBF14}.
Compared to POD (Proper Orthogonal Decomposition) \cite[Ch.~9.1]{morAnt05},
which in this context is equivalent to using solely the controllability subspace (basis) $U_C$ as a Galerkin projection,
DSPMR incorporates controllability \emph{and observability} information.
Yet, in comparison to balanced POD \cite{morWilP02,morRow05,morOrSK12},
the truncated controllability and observability subspaces $U_C$, $U_O$ are not balanced, but directly concatenated.

An extension to the DSPMR method is also proposed in \cite{morPen06}, called \emph{Refined} Dominant Subspace Projection Model Reduction.
The eponymous refinement is given by weighting factors $\omega_C, \omega_O > 0$ for the controllability and observability subspace bases respectively.
The weighting factors are selected as the Frobenius norm of the respective low-rank factors, $\omega_C := \|Z_C\|_{\text{F}}^{-1}$ and $\omega_O := \|Z_O\|_{\text{F}}^{-1}$, yielding:
\begin{align*}
 QR &\stackrel{\text{QR}}{=} \begin{bmatrix} (\omega_C Z_C) & (\omega_O Z_O)\end{bmatrix} \rightarrow U_1 := Q, \\
 UDV^\T &\stackrel{\text{SVD}}{=} \begin{bmatrix}(\omega_C Z_C) \,\,\, (\omega_O Z_O)\end{bmatrix} \rightarrow U_1 := U.
\end{align*}
Obviously, this is only sensible for the Cholesky factor variant, as the norm of the (orthonormal) singular vectors is one.

The weighting normalizes the system Gramian factors.
This normalization equilibrates the influence of controllability ($W_C$ depends only on $\{A,B\}$) and observability ($W_O$ depends only on $\{A,C\}$),
which may be skewed, i.e., due to different scaling of $B$ and $C$.
A similar idea for combining weighted subspaces is also used in the cotangent lift method from \cite{morPenM16}.

\subsection{Cross-Gramian-Based Dominant Subspaces}
Instead of the controllability and observability Gramians,
also the cross Gramian can be used to obtain a dominant subspace projection.
A truncated SVD of the cross Gramian (based on a pre-selected rank or approximation error),
\begin{align}\label{eq:wxsvd}
 W_X \stackrel{\text{tSVD}}{=} U_X D_X V_X^\T,
\end{align}
produces left and right singular vectors aggregated in matrices $U_X$ and $V_X$,
which induce subspaces associated to controllability ($U_X$) and observability ($V_X$) of the underlying system $(A,B,C,E)$ \cite[Sec.~B]{morWon08}.

In \cite[Sec.~4.3]{morSorA02}, it is noted, that the sole use of either, $U_X$ or $V_X$, as a Galerkin projection,
will largely omit observability or controllability information respectively.
Hence, both subspaces should be incorporated in the reducing and lifting operator.
Balanced truncation, for example, determines a suitable Petrov-Galerkin projection\footnote{Balanced truncation yields a Galerkin projection for state-space symmetric systems, \linebreak $A = A^\T, B = C^\T, E = E^\T$ \cite{morBruCT02}.},
where $U_1 \neq V_1$, by simultaneous diagonalization of the controllability and observability Gramians,
while approximate balancing applies the left and right singular vectors of the cross Gramian as oblique projections directly \cite{morRahVA14}.

For the proposed variant of the dominant subspace method (for an algorithmic description see \cref{sec:algo}),
the left and right singular vectors are conjoined as before,
but also scaled \emph{column-wise} by the associated singular values:
\begin{align*}
 \begin{bmatrix}(U_X D_X) & (V_X D_X)\end{bmatrix} &\stackrel{\text{SVD}}{=} U_{CO} D_{CO} V_{CO}^\T \rightarrow U_1 := U_{CO}.
\end{align*}
Here, the singular values are used to scale the singular vectors,
since the majorization property \cite[Remark~2.1]{morSorA02} relates the singular values of the cross Gramian with the (absolute value of the) cross Gramian's eigenvalues,
which in turn are equal to the Hankel singular values of a symmetric system \eqref{eq:hsv}.
So, instead of normalizing the controllability and observability subspaces (as a whole), as in refined DSPMR,
based on the common controllability-observability measure, the singular values of the cross Gramian,
the vectors spanning the compound subspace are scaled individually.
Here explicitly a rank-revealing SVD is used, instead of a QR decomposition,
as the singular values $D_{CO}$ will be used for an error indicator in \cref{sec:staberr}.
An advantage of the cross-Gramian-based dominant subspace projection method is this common measure of minimality \cite{morFerN82},
the singular values $\sigma_i = D_{CO,ii}$ associated jointly to the ``controllability'' and ``observability'' subspaces.

\subsubsection{Algorithmic Computation}\label{sec:algo}
The computation of the proposed cross-Gramian-based dominant subspace projection,
as well as the classic dominant subspace projection consists of two phases:
First, the computation of the system Gramians, either the cross Gramian,
or the controllability and observability Gramians.
And second, the assembly of the reducing (and lifting) operator.

For large-scale systems, the computation of dense system Gramians,
which are of dimension $N \times N$, may be infeasible or at least inefficient.
To this end, low-rank representations of the Gramians can be computed,
for the cross Gramian, in example by the implicitly restarted Arnoldi algorithm \cite{morSorA02},
the factorized iteration \cite{Ben04}, a factored ADI \cite{BenK14} or a low-rank empirical cross Gramian \cite{morHimLRetal17}.

Overall, the \emph{cross-Gramian-based dominant subspace} algorithm is summarized by:

\begin{enumerate}

 \item Compute (low-rank) cross Gramian:
  \begin{enumerate}
   \item As solution to a matrix equation: $A W_X E + E W_X A = -B C$,
   \item or by quadrature: $W_X = \int_0^\infty \e^{E^{-1} At} E^{-1} BC \e^{E^{-1} At} E^{-1} \D t$.
  \end{enumerate}

 \item Compute (truncated) SVD of the cross Gramian:
  \begin{itemize}
   \item[] $U_X D_X V_X^\T \stackrel{\text{tSVD}}{=} W_X$.
  \end{itemize}

 \item Compute (rank-revealing) SVD of conjoined and weighted left and right singular vectors:
  \begin{itemize}
   \item[] $U_1 D_1 V_1 \stackrel{\text{SVD}}{=} \begin{bmatrix} (U_X D_X) & (V_X D_X) \end{bmatrix}$.
  \end{itemize}

 \item Apply left singular vectors to system matrices following \eqref{eq:rom}:
  \begin{itemize}
   \item[] $A_r := U_1^\T A U_1, \quad B_r := U_1^\T B, \quad C_r := C U_1, \quad E_r := U_1^\T E U_1$.
  \end{itemize}

\end{enumerate}

The Galerkin projection $U_1$ is the cross-Gramian-based dominant subspace projection.
In principle, a similar procedure can be conducted using controllability and observability Gramians,
yet it is not immediately clear if the SVD of the (weighted) conjoined singular vectors $\begin{bmatrix} (U_C D_C) & (U_O D_O) \end{bmatrix}$ yields an equally useful measure.

The efficiency of computing a low-rank approximation of $W_X$ depends on the rank of $BC$ and the symmetry of $A$.
Usually, this means, the more (linearly) independent inputs and outputs a system has,
and the less symmetric a system matrix is,
the higher the rank of the approximated cross Gramian.

\subsection{Error Indicator}\label{sec:staberr}
In this section an error indicator for the cross-Gramian-based dominant subspace method is developed.
Previous works, such as \cite{morShiS05,morWolPL11,morRedK18,morWanY18}, already introduced error \emph{bounds} for the Hardy $H_2$-norm.
Here, an $H_2$-error \emph{indicator} of simple structure using time-domain quantities is proposed,
which is loosely related to the simplified balanced gains approach from \cite{morDav86}.
The $H_2$-norm is particularly interesting, since an error estimation has relevance for the frequency-domain and the time-domain \cite[Ch.~2]{Tos13},
and it also describes the energy ($L_2$-norm) of the system's impulse response.
Before this error indicator is derived, a straight-forward property of the matrix exponential is presented.

\begin{lemma}\label{thm:matexp}
Given matrices $A \in \R^{N \times N}$ and $U \in \R^{N \times n}$, $n \leq N$,
the following holds:
\begin{align*}
 U \e^{U A U^\T} U^\T = UU^\T \e^{A UU^\T} = \e^{UU^\T A} UU^\T.
\end{align*}
\end{lemma}

\begin{proof}
The proof is a trivial consequence on the associativity of the matrix product.
\begin{align*}
 U \e^{U^\T A U} U^\T &= U (\sum_{k=0}^\infty \frac{1}{k!} (U^\T A U)^k) U^\T \\
                      &= U (I + (U^\T A U) + \frac{1}{2} (U^\T A U)(U^\T A U) + \dots) U^\T \\
                      &= UU^\T (I + A UU^\T + \frac{1}{2} A UU^\T A UU^\T + \dots) \\
                      &= UU^\T \e^{A UU^\T}.\qedhere
\end{align*}
\end{proof}

Next, the error indicator is constructed, which is derived from the $L_2$-norm of the impulse response error system,
and we assume, for ease of exposition but without loss of generality, $E = I$:
	\begin{align*}
 \begin{pmatrix}\dot{x}(t) \\ \dot{x}_r(t) \end{pmatrix} &= \begin{pmatrix} A & 0 \\ 0 & A_r \end{pmatrix} \begin{pmatrix} x(t) \\ x_r(t) \end{pmatrix} + 
 \begin{pmatrix} B \\ B_r \end{pmatrix} u(t) \\
 y_e(t) &= \begin{pmatrix} C & -C_r \end{pmatrix} \begin{pmatrix} x(t) \\ x_r(t) \end{pmatrix}.
\end{align*}
We consider only SISO systems for this error indicator and unit impulse (Dirac impulse) inputs \mbox{$u(t) \equiv \delta(t)$}, defined by the properties:
\begin{align}\label{eq:delta}
 \int \delta(t) \D t = 1, \quad \delta(t \neq 0) = 0.
\end{align}

First, the $H_2$-norm of the error system, in impulse response form, is transformed in a manner so that \cref{thm:matexp} can be applied.
Note, that the error system of a SISO system is also a SISO system with a scalar and thus symmetric impulse response:
\begin{align*}
\|y_e\|_{L_2}^2 &= \tr\Big(\int_0^\infty (\begin{pmatrix} C & -C_r \end{pmatrix} \begin{pmatrix} \e^{A t} & 0 \\ 0 & \e^{A_r t} \end{pmatrix} \begin{pmatrix} B \\ B_r \end{pmatrix})^2 \D t\Big) \\
                &= \tr\Big(\int_0^\infty (C \e^{At} B - C_r \e^{A_r t} B_r)^2 \D t\Big) \\
                &= \tr\Big(C \int_0^\infty \e^{At} BC \e^{At} - \e^{At} B C_r \e^{A_r t} U_1^\T \\
                                                              &\hskip10.2em - U_1 \e^{A_r t} B_r C \e^{A t} + U_1 \e^{A_r t} B_r C_r \e^{A_r t} U_1^\T \D t B\Big), \\
\intertext{applying the definition of the reduced quantities \eqref{eq:rom}, and subsequently the result of \cref{thm:matexp}, gives:}
\|y_e\|_{L_2}^2 &= \tr\Big(C \int_0^\infty \e^{At} BC \e^{At} - \e^{At} BC \e^{U_1 U_1^\T A t} U_1 U_1^\T \\
                                                              &\hskip10.2em - U_1 U_1^\T \e^{A U_1 U_1^\T t} BC \e^{At} \\
                                                              &\hskip10.2em + U_1 U_1^\T \e^{A U_1 U_1^\T t} BC \e^{U_1 U_1^\T A} U_1 U_1^\T \D t B\Big).
\end{align*}
The next step is approximating the matrix exponentials $\e^{AU_1 U_1^\T t}$ and $\e^{U_1 U_1^\T A t}$ by the homogeneous system's solution operator,
\begin{align*}
 \e^{AU_1 U_1^\T t} \approx \e^{At}, \quad \e^{U_1 U_1^\T A t} \approx \e^{At},
\end{align*} 
which allows to factor the previous representation to:
\begin{align*}
\|y_e\|_{L_2}^2 \approx \tr\Big(C \int_0^\infty (I - U_1 U_1^\T) (\e^{At} BC \e^{At}) (I - U_1 U_1^\T) \D t B\Big).
\end{align*}
Now, we move the projection error terms $(I - U_1 U_1^\T)$ out of the integral,
identify the resulting expression with the cross Gramian $W_X$, and exploit the cyclic permutability of the trace argument:
\begin{align*}
\|y_e\|_{L_2}^2 &\approx \tr\Big(C (I - U_1 U_1^\T) \int_0^\infty \e^{At} BC \e^{At} \D t (I - U_1 U_1^\T) B\Big) \\
                &= \tr\Big((I - U_1 U_1^\T) W_X (I - U_1 U_1^\T) BC\Big).
\end{align*}
The (full) SVD of the cross Gramian is given by adding to its truncated SVD, $W_X \stackrel{\text{tSVD}}{=} U_X D_X V_X^\T$,
(the SVD of) its truncated remainder:
\begin{align*}
 W_X \stackrel{\text{SVD}}{=} U_X D_X V_X^\T + U_2 D_2 V_2^\T.
\end{align*}
Together with an observation on the truncated SVD's singular vectors:
\begin{align*}
 U_X &= \begin{bmatrix}(U_X D_X) & (V_X D_X)\end{bmatrix} \begin{bmatrix} D_X^{-1} \\ 0 \end{bmatrix} = U_1 D_1 V_1^\T \begin{bmatrix} D_X^{-1} \\ 0 \end{bmatrix}, \\
 V_X &= \begin{bmatrix}(U_X D_X) & (V_X D_X)\end{bmatrix} \begin{bmatrix} 0 \\ D_X^{-1} \end{bmatrix} = U_1 D_1 V_1^\T \begin{bmatrix} 0 \\ D_X^{-1} \end{bmatrix},
\end{align*}
the following simplification entails:
\begin{align*}
 & (I - U_1 U_1^\T) W_X (I - U_1 U_1^\T) \\
=& (I - U_1 U_1^\T) (U_X D_X V_X^\T + U_2 D_2 V_2^\T) (I - U_1 U_1^\T) \\
=& (I - U_1 U_1^\T) (U_2 D_2 V_2^\T) (I - U_1 U_1^\T).
\end{align*}
Next, the \textsc{von Neumann}'s trace inequality \cite{Mir75},
which assumes (without loss of generality) descendingly ordered singular values $\sigma_k(\cdot) \geq \sigma_{k+1}(\cdot)$ is applied,
followed by the Cauchy-Schwarz inequality (with $\sigma(\cdot)$ being the vector of singular values):
\begin{align*}
\|y_e\|_{L_2}^2 &\approx \tr\Big((I - U_1 U_1^\T) (U_2 D_2 V_2) (I - U_1 U_1^\T) BC\Big) \\
                &\leq \sum_{k=1}^N \sigma_k(D_2) \sigma_k(BC) = \langle \sigma(D_2), \sigma(BC) \rangle \\
                &\leq \|\sigma(D_2)\|_2 \|\sigma(BC)\|_2 = \|D_2\|_F \,\, \|BC\|_F.
\end{align*}
Since the singular values of $D_2$ correspond to the truncated tail of the cross Gramian's singular values,
and $BC$ is of rank one, due to the SISO nature of the system, we obtain:
\begin{align}\label{eq:errin}
  \|y_e\|_{L_2}^2 &\lessapprox \|BC\|_2 \sqrt{\sum_{k=n+1}^N \sigma_k^2(W_X)}.
\end{align}
Note, that $\|BC\|_2 = \|B\|_2 \|C\|_2$, as $B$ and $C^\T$ are column vectors.
Overall, this derivation yields the following error indicator:

\begin{errorin}
The $L_2$ impulse response model reduction error for a cross-Gramian-based dominant subspaces reduced order model is approximated by:
\begin{align}\label{eq:errind}
 \|y - \tilde{y}\|_{L_2} &\lessapprox \sqrt{\|B\|_2 \|C\|_2 \sqrt{\sum_{k=n+1}^N \sigma_k^2(W_X)}}.
\end{align}
\end{errorin}

One might assume that the spectral norm representation of the error indicator would be more convenient,
yet in \cref{sec:fc} we will show the advantage of the final Frobenius norm form.

\begin{remark}
Using the Cauchy-Schwarz inequality as in \cite{morGugAB08},
this impulse response error indicator can be extended to squarely integrable inputs $u \in L_2$.
\end{remark}

This error indicator could be extended directly to square MIMO systems,
yet the Frobenius norm estimation could not be used anymore,
which is essential to the practical computation detailed in \cref{sec:fc}.
Alternatively, it extends to any MIMO systems by either using the averaged system $(A,\sum_{i=1}^M b_i, \sum_{j=1}^Q c_j^\T)$
associated to the non-symmetric cross Gramian \eqref{eq:avgsys} \cite{morHimO16},
or, selecting a SISO sub-system $(A,b_k,c_\ell^\T)$, for example based on:
\begin{align*}
  (k,\ell) = \argmax_{i,j} \langle |b_i|, |c_j| \rangle.
\end{align*}
Furthermore, the error indicator holds also for systems with $E \neq I$, $E > 0$:
\begin{align*}
 \|y - \tilde{y}\|_{L_2} &\lessapprox \sqrt{\|E^{-1} B\|_2 \|C\|_2 \sqrt{\sum_{k=n+1}^N \sigma_k^2(W_X)}},
\end{align*}
which follows from: $E\dot{x}(t) = Ax(t) + Bu(t) \Rightarrow \dot{x}(t) = E^{-1}Ax(t) + E^{-1}Bu(t)$.

\subsection{Fused Computation}\label{sec:fc}
Even for moderately sized systems, the computation of the (cross) Gramian's singular vectors may be a computationally challenging task\footnote{For the presented numerical examples the SVDs of system Gramians comprises the dominant fraction of computation time.}.
To compute the dominant subspace projections from the cross Gramian, or the controllability and observability Gramians,
the hierarchical approximate proper orthogonal decomposition (HAPOD) \cite{morHimLR18a} is used.

The HAPOD enables a swift computation of left singular vectors of arbitrary partitioned data sets,
based on a selected projection error (on the input data) $\varepsilon > 0$ and a tree hierarchy with the data (Gramian) partitions as leafs.
The tree hierarchy utilized for the experiments in this work is given by a combination of special topologies discussed in \cite{morHimLR18a},
the incremental HAPOD (maximally unbalanced binary tree) and the distributed HAPOD (star).
Two incremental HAPODs are performed for the Gramian partitions respectively
and subsequently a distributed HAPOD of the resulting singular vectors from both sub-trees yields the dominant subspace projection.
\cref{fig:hapod} illustrates the overall HAPOD tree.

Since the HAPOD computes only left singular vectors,
but the right singular vectors of the cross Gramian are also needed,
the HAPOD of the cross Gramian (left singular vectors) and the transposed cross Gramian (right singular vectors) is computed.
In the following numerical examples, the (full-order) empirical linear cross Gramian \cite[Sec.~3.1.3]{morHim18b} is used,
as in-memory storage of the Gramian(s) is possible.
For settings, where only parts of the cross Gramian can be kept in memory,
the low-rank empirical cross Gramian \cite{morHimLRetal17} for the left singular vectors,
and the low-rank empirical cross Gramian of the \textbf{adjoint system} for the right singular vectors, can be utilized,
since the cross Gramian of the adjoint system is equal to the system's transposed cross Gramian:
\begin{align*}
 \widetilde{W}_X &:= \int_0^\infty \e^{E^{-\T}A^\T t} E^{-\T} C^\T B^\T \e^{E^{-\T} A^\T t} E^{-\T} \D t \\
                 &= \int_0^\infty (C E^{-1}\e^{A E^{-1} t})^\T (E^{-1} \e^{A E^{-1} t} B)^\T \D t \\
                 &\hskip-1.1em\stackrel{\text{Lemma~1}}{=} \int_0^\infty (C \e^{E^{-1} At} E^{-1})^\T (\e^{E^{-1} At} E^{-1} B)^\T \D t \\
                 &= \int_0^\infty (\e^{E^{-1} At} E^{-1} BC \e^{E^{-1} At} E^{-1})^\T \D t = W_X^\T.
\end{align*}

\begin{figure}\centering
\includegraphics[width=.82\textwidth]{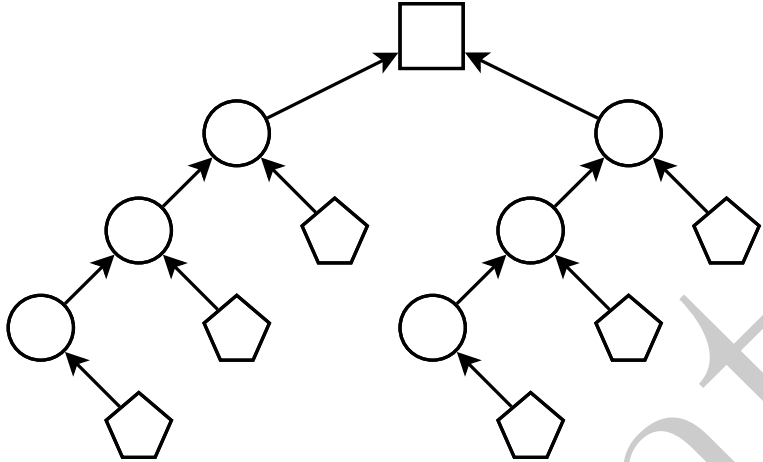}
\caption{HAPOD tree topology for the cross-Gramian-based dominant subspaces method.
Pentagons symbolize partitions of the cross Gramian (left) and the adjoint cross Gramian (right), respectively.
Circles mark sub-PODs, while a square represents the root-POD returning the overall (HA)POD described in \cref{sec:fc}.
In the context at hand, PODs correspond to SVDs.}
\label{fig:hapod}
\end{figure}

Since a projection-error-driven SVD method is used,
the following error bound holds\footnote{This is shown for the HAPOD in \cite{morHimLR18a}.
Specifically, the mean $L_2$ projection error, $\frac{1}{N} \sum_{k=1}^N \|(I-U_1 U_1^\T) W_{X,*k}\|^2 \leq \varepsilon^2$, is bounded by the HAPOD,
which has to be taken into account for the practical computation.},
for a given projection error $\varepsilon>0$:  
\begin{align*}
 \sum_{k=1}^N \|(I-U_1 U_1^\T) W_{X,*k}\|^2 = \sum_{k=n+1}^N \sigma_k^2(W_X) \leq \varepsilon^2.
\end{align*}
This means the error indicator \eqref{eq:errin} can be bounded using the prescribed cross Gramian's projection error,
\begin{align}\label{eq:apriori}
 \|y - \tilde{y}\|_{L_2} &\lessapprox \sqrt{\|B\|_2 \|C\|_2 \sqrt{\sum_{k=n+1}^N \sigma_k^2(W_X)}} \leq \sqrt{\varepsilon \|B\|_2 \|C\|_2},
\end{align}
thus making it an \emph{a-priori} error indicator.
This approximate error prediction for a given projection error $\varepsilon$
and the Euclidean norms (spectral norms) of the input and output operators, without computing any system Gramians, is the main advantage of this method.
The tightness of this error indicator is evaluated in the following numerical results.

\section{Numerical Results}\label{sec:nr}
Following, two numerical examples are presented to illustrate the previous findings.
These numerical experiments are conducted using MATLAB~2018a \cite{matlabweb}.
The system Gramians needed for the dominant subspace methods,
the controllability and observability Gramian for plain and refined DSPMR as well as balanced truncation,
and the cross Gramian for the cross-Gramian-based dominant subspaces,
are computed as empirical Gramians \cite{morHim18b} using \texttt{emgr} -- empirical Gramian framework in version 5.7 \cite{morHim19a}.
All simulated trajectories for the construction of these \emph{empirical dominant subspaces} are computed using the implicit Euler method,
and the HAPOD is computed via \cite{morHimR19}.

\subsection{FOM Benchmark}\label{sec:lint}
The first numerical example compares the cross-Gramian-based dominant subspace method with the classic unrefined and refined dominant subspace method\footnote
{The refined DSPMR method is computed with weighting coefficients $\omega_C =  \frac{\|Z_O\|_F}{\|Z_C\|_F}$, \linebreak $\omega_O = 1$ of the controllability and observability factors respectively for numerical reasons.}
as well as (empirical) balanced truncation\footnote{In the numerical experiments at hand,
low-rank Gramians are balanced, whereas the rank is determined by the projection error of the POD compression of the empirical controllability and observability Gramian.
In this sense, this method is related to balanced POD \cite{morRow05}.}
for the ``FOM'' example in \cite{morPen06}, which is also part of the SLICOT Benchmark Collection \cite{morChaV02}.
This linear SISO system (with $E = I$) of the structure:
\begin{align*}
 \dot{x}(t) &= Ax(t) + Bu(t), \\
       y(t) &= Cx(t),
\end{align*}
is of order $N = 1006$, and the system components are given by:
\begin{align*}
 A_1 &= \begin{pmatrix} -1 & 100 \\ -100 & -1 \end{pmatrix}, \quad
 A_2 = \begin{pmatrix} -1 & 200 \\ -200 & -1 \end{pmatrix}, \quad
 A_3 = \begin{pmatrix} -1 & 400 \\ -400 & -1 \end{pmatrix}, \\
 A_4 &= \begin{pmatrix} -1 \\ & -2 \\ & & \ddots \\ & & & -1000 \end{pmatrix}, \quad
 A = \begin{pmatrix} A_1 \\ & A_2 \\ & & A_3 \\ & & & A_4 \end{pmatrix}, \\
 C &= \begin{pmatrix} C_1 & C_2 \end{pmatrix}, \quad
 C_1 = \begin{pmatrix} 10 & \dots & 10 \end{pmatrix} \in \R^{6}, \quad
 C_2 = \begin{pmatrix} 1 & \dots & 1 \end{pmatrix} \in \R^{1000}, \\
 B &= C^\T.
\end{align*}
The empirical Gramians are constructed using random binary input,
and the reduced systems are tested with impulse input, to evaluate the error indicator.

In \cref{fig:numex1}, the (empirical) balanced truncation, the (empirical) dominant subspaces method,
the (empirical) refined dominant subspaces method, the (empirical) cross-Gramian-based dominant subspaces method, the predicted error \eqref{eq:apriori} and the error indicator \eqref{eq:errind} are compared,
for a given projection error $\varepsilon \in \{ 10^{-3}, \dots, 10^{-12} \}$ of the utilized empirical cross Gramians.
The same projection error is selected for the controllability and observability Gramians used by the unrefined DSPMR, refined DSPMR and low-rank empirical balanced truncation.

\begin{figure}\centering

\begin{subfigure}{.9\textwidth}
 \includegraphics[width=\textwidth]{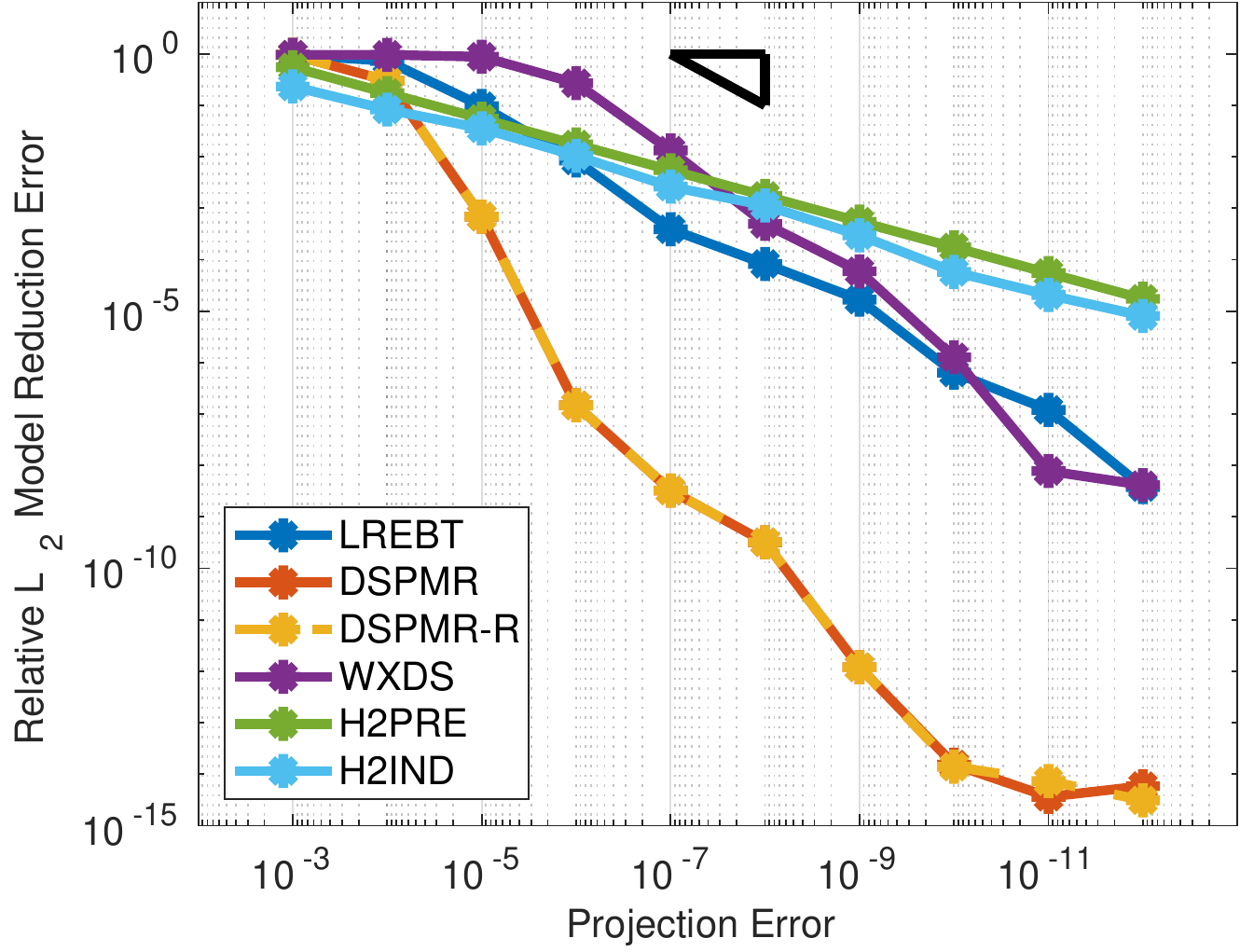}
 \caption{Cross Gramian projection error versus $L_2$ model reduction error.}
 \label{fig:numex1a}
\end{subfigure}

\begin{subfigure}{.9\textwidth}
 \includegraphics[width=\textwidth]{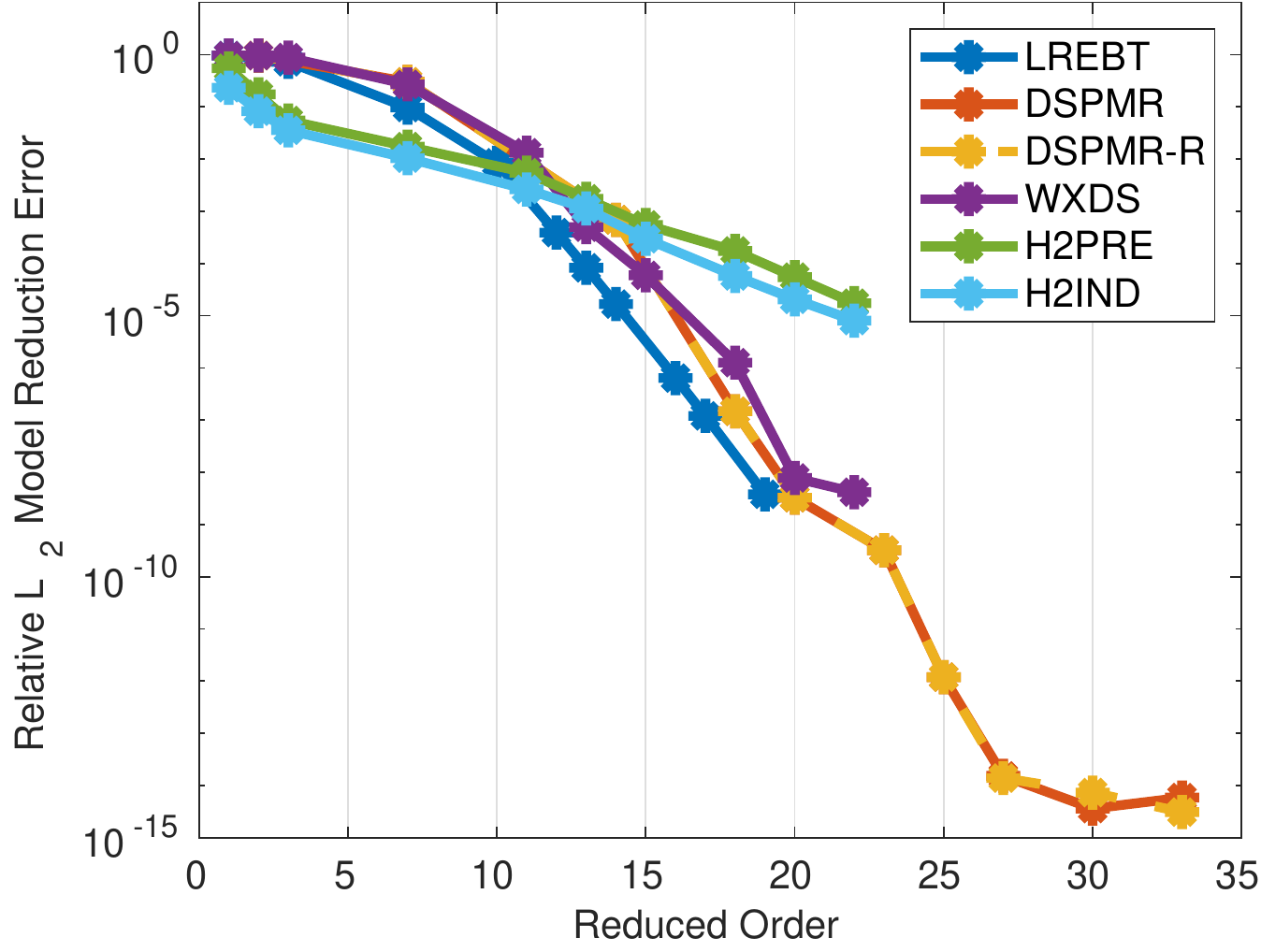}
 \caption{Reduced order versus $L_2$ model reduction error.}
 \label{fig:numex1b}
\end{subfigure}
\caption{Model reduction error of the \textbf{FOM benchmark} example from \cref{sec:lint} for low-rank empirical balanced truncation (LREBT),
         dominant subspaces (DSPMR), refined dominant subspaces (DSPMR-R), cross-Gramian-based dominant subspaces (WXDS), the predicted $H_2$-error (H2PRE) and the $H_2$-error indicator (H2IND).}
\label{fig:numex1}
\end{figure}

In \cref{fig:numex1a} the prescribed projection error of the respective Gramians is plotted against the resulting relative $L_2$ model reduction error.
For a given projection error, the refined DSPMR and DSPMR-R method produce the lowest model reduction error, and low-rank balanced truncation the largest,
while the proposed cross-Gramian-based dominant subspace method is in-between.
The error indicator overestimates the error for larger and underestimates for smaller projection errors,
the predicted error is reasonably close to the error indicator.
Note, that the error indicator just scales the projection error by a constant,
hence it appears as a line in the log-log plot.

\cref{fig:numex1b} depicts the resulting reduced order of the tested methods against the model reduction error.
Balanced truncation produces the smallest, and DSPMR, DSPMR-R the largest reduced models, again the cross-Gramian-based method is in-between.
These results follow intuitions that DSPMR produces the most accurate, but largest subspaces,
while balanced truncation may have a smaller, and hence less accurate subspaces.
Hence, the cross-Gramian-based dominant subspace method appears as a compromise.
The error indicator is rather coarse, which is due to its simple structure.

\subsection{Convection Benchmark}\label{sec:conv}
The second numerical example evaluates the convection benchmark \cite[Convection]{morwiki}\footnote{\url{http://modelreduction.org/index.php/Convection}}
from the Oberwolfach Benchmark Collection \cite{morMooG05}.
This is a two-dimensional computational fluid dynamics application of thermal flow modeled by a convection-diffusion partial differential equation:
\begin{align*}
 \frac{\partial T}{\partial t} = \kappa \nabla^2 T -  v \nabla T + \dot{q}
\end{align*}
with the solution temperature $T(x,t)$, the thermal conductivity $\kappa$, the fluid speed $v$ of fixed direction, and the heat generation rate $\dot{q}$.
The model is discretized in space using the finite element method, yielding a generalized linear system \eqref{eq:sys} of order $N = 9669 \approx 10^4$,
a single input $M = 1$, five outputs $Q = 5$ and $E \neq I$.
For a more detailed description of this benchmark see \cite{morMooG05} and references therein.
This model is tested in two variants:
First, in a symmetric setting with zero flow speed $v = 0$, and second,
in a non-normal setting with a flow speed $v = 0.5$.
Due to the MIMO nature of the system we use the average system (see \eqref{eq:avgsys}) for the error indicator computation.

This set of experiments is organized in the same manner as \cref{sec:lint},
but conducted for the prescribed projection errors $\varepsilon \in \{ 10^{-2}, \dots, 10^{-8} \}$.
As indicated in \cref{sec:staberr},
the average system (averaged over outputs) $\bar{C} := \sum_{q=1}^Q c_q^\T$ is used for the computation of the error indicator. 
The resulting reduced order models are tested with impulse input $u(t) = \delta(t)$.

\subsubsection{Symmetric Variant}
The experimental results of the symmetric variant ($v = 0$) are depicted in \cref{fig:numex2}.
In \cref{fig:numex2a} the prescribed projection error for the (empirical) system Gramians versus the resulting relative $L_2$ model reduction error is plotted.
As in \cref{sec:lint}, the DSPMR method produces the reduced order models with the lowest model reduction error.
The refined DSPMR exhibits slightly larger model reduction errors for small projection errors, otherwise it is following the plain DSPMR method.
Reduced systems from (empirical) balanced truncation and the (empirical) cross-Gramian-based dominant subspace method result in similar errors,
while the error indicator behaves like a upper bound to the cross-Gramian-based model reduction error.

\begin{figure}\centering

\begin{subfigure}{.9\textwidth}
 \includegraphics[width=\textwidth]{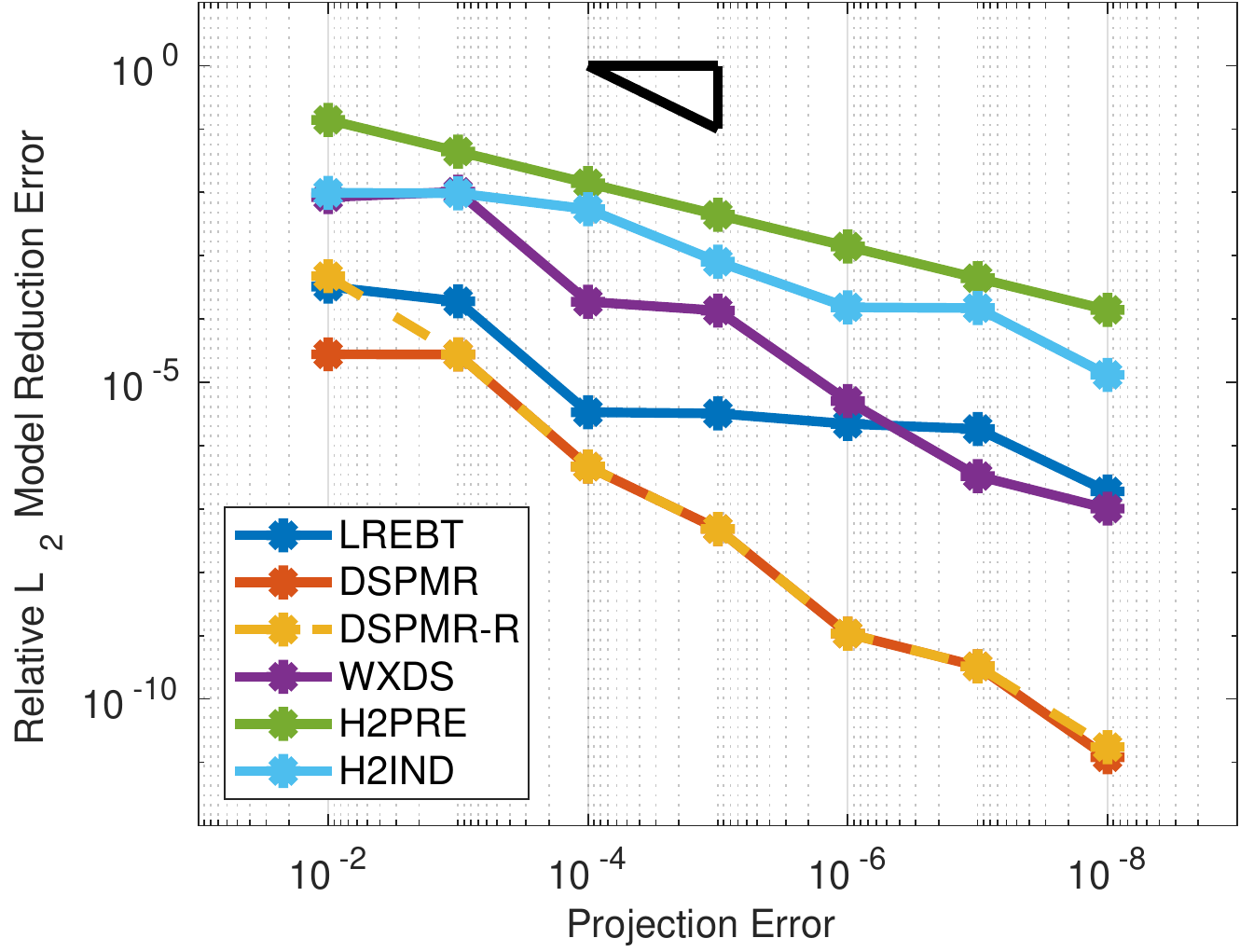}
 \caption{Cross Gramian projection error versus $L_2$ model reduction error.}
 \label{fig:numex2a}
\end{subfigure}
\begin{subfigure}{.9\textwidth}

 \includegraphics[width=\textwidth]{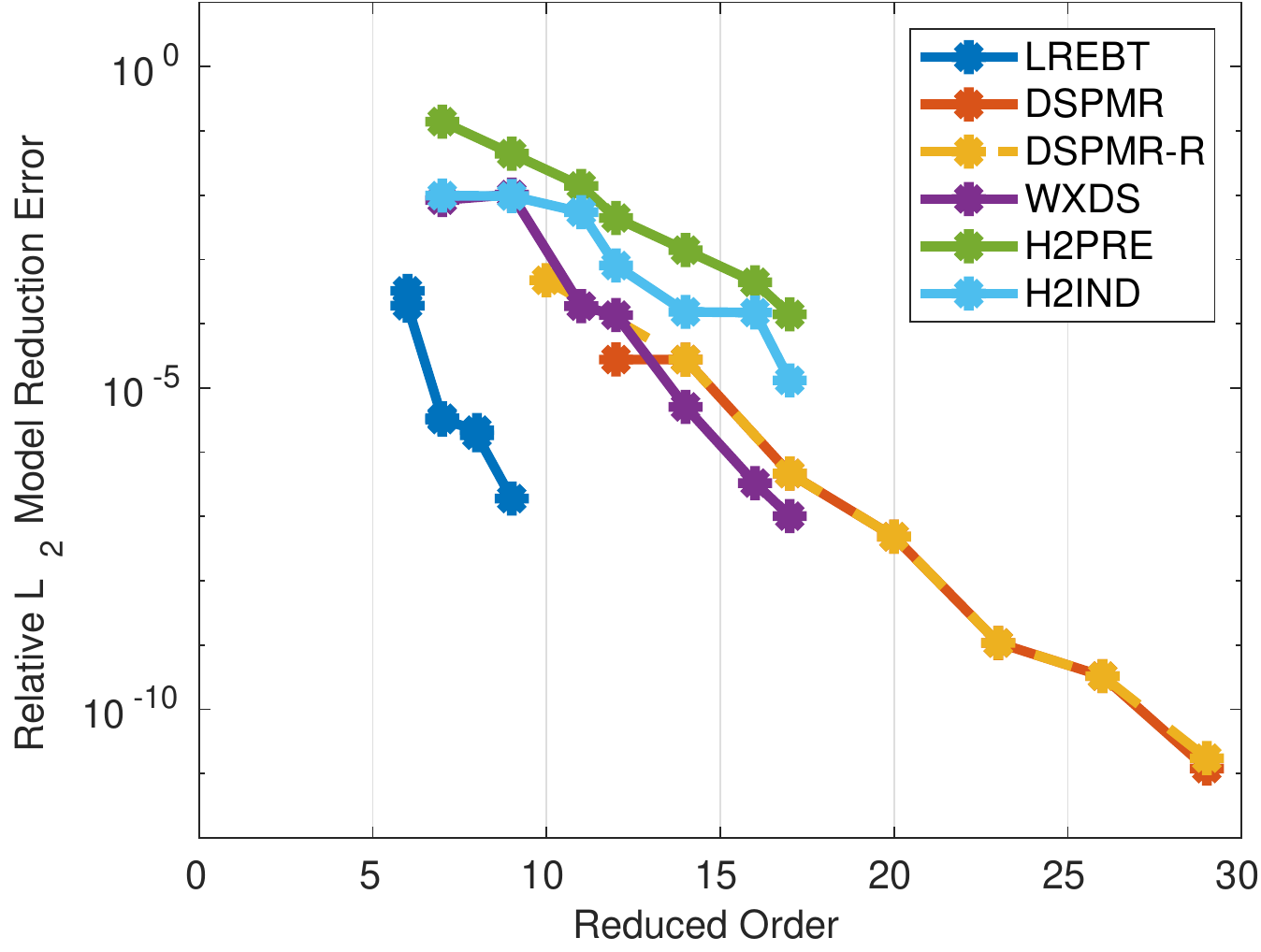}
 \caption{Reduced order versus $L_2$ model reduction error.}
 \label{fig:numex2b}
\end{subfigure}
\caption{Model reduction error of the \textbf{symmetric} convection benchmark from \cref{sec:conv} for low-rank empirical balanced truncation (LREBT),
         dominant subspaces (DSPMR), refined  dominant subspaces (DSPMR-R), cross-Gramian-based dominant subspaces (WXDS), the predicted $H_2$-error (H2PRE) and the $H_2$-error indicator (H2IND).}
\label{fig:numex2}
\end{figure}

\cref{fig:numex2b} shows the model reduction error for the reduced orders resulting from the prescribed projection error.
Balanced truncation achieves the smallest and DSPMR the largest reduced models,
the cross-Gramian-based dominant subspace method reduced order model dimension lies in between,
and the error indicator shows a similar behavior as the latter.

\subsubsection{Non-Normal Variant}
The experimental results of the non-normal variant ($v = 0.5$) are presented in \cref{fig:numex3}.
Overall, the plots \cref{fig:numex3a} and \cref{fig:numex3b} are similar to the symmetric variant,
with a reasonably close predicted error, which is equal for symmetric and non-normal benchmark variants.
Yet, in case of the non-normal benchmark variant, the error indicator is not as tight, compared to the symmetric variant.

\begin{figure}\centering

\begin{subfigure}{.9\textwidth}
 \includegraphics[width=\textwidth]{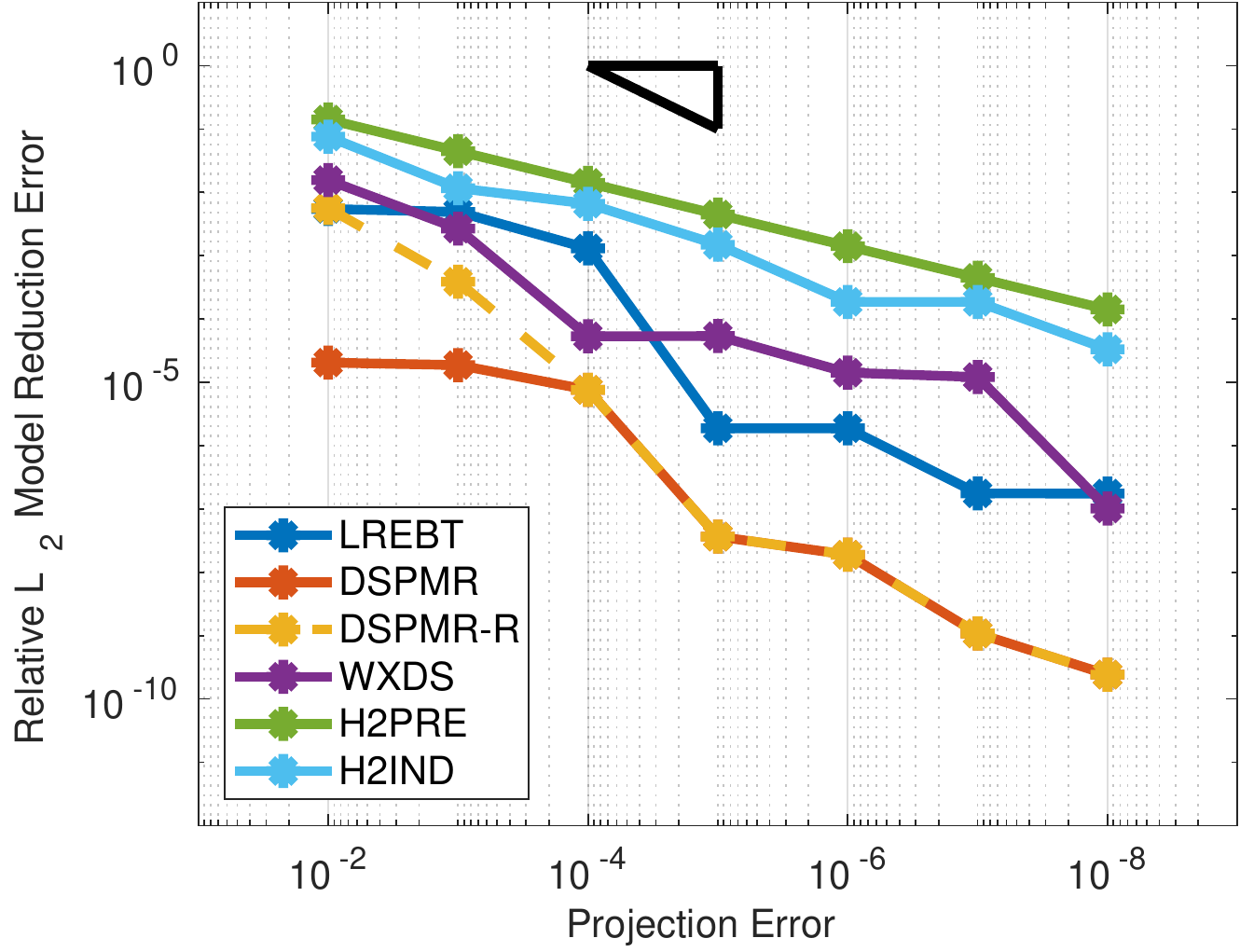}
 \caption{Cross Gramian projection error versus $L_2$ model reduction error.}
 \label{fig:numex3a}
\end{subfigure}

\begin{subfigure}{.9\textwidth}
 \includegraphics[width=\textwidth]{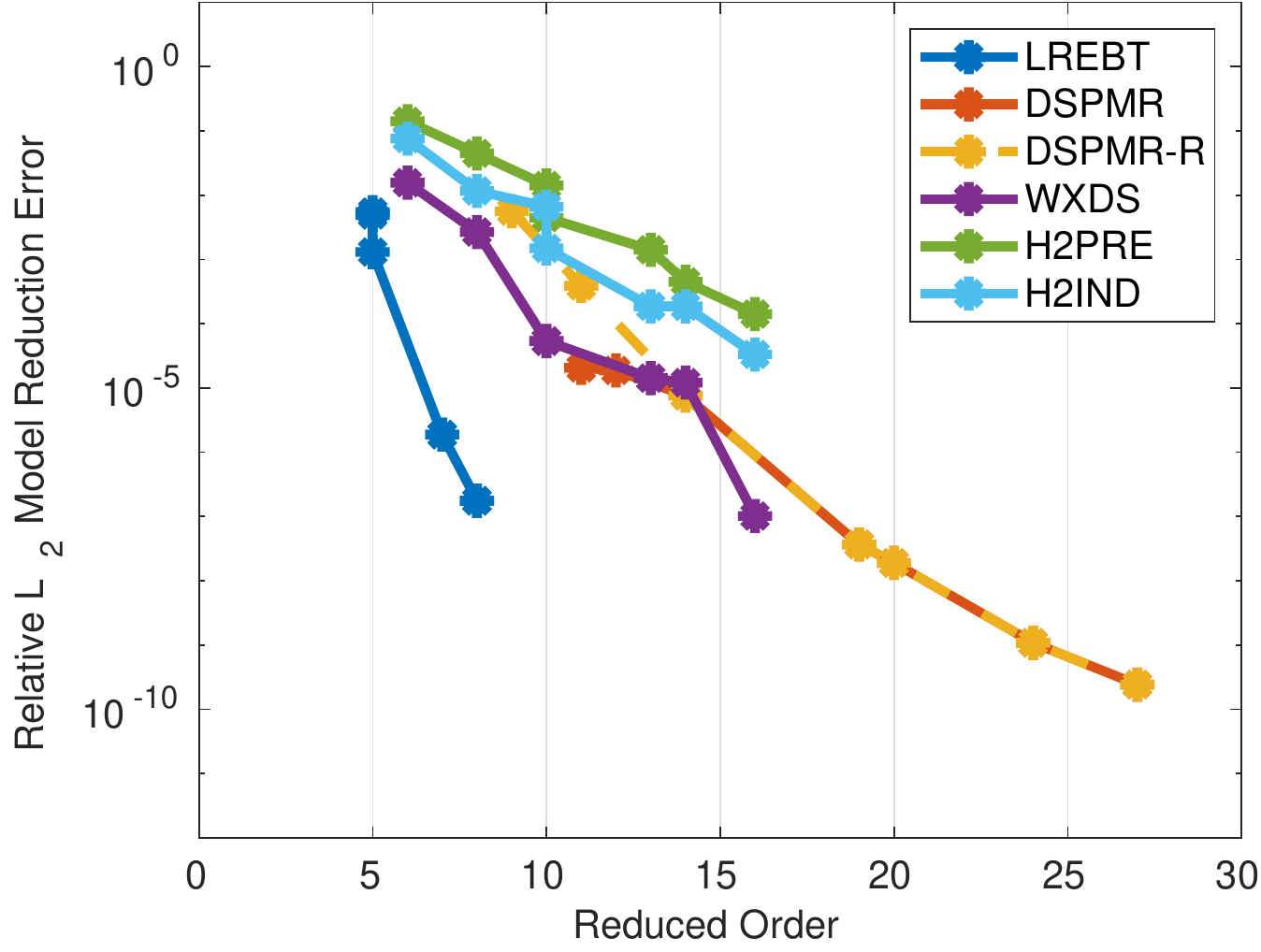}
 \caption{Reduced order versus $L_2$ model reduction error.}
 \label{fig:numex3b}
\end{subfigure}
\caption{Model reduction error of the \textbf{non-normal} convection benchmark from \cref{sec:conv} for low-rank balanced truncation (LREBT),
        dominant subspaces (DSPMR), refined  dominant subspaces (DSPMR-R), cross-Gramian-based dominant subspaces (WXDS), the predicted $H_2$-error (H2PRE) and the $H_2$-error indicator (H2IND).}
\label{fig:numex3}
\end{figure}

\section{Summary}\label{sec:conc}
In this work we revisited the dominant subspaces projection model reduction method,
and presented a variant based on the cross Gramian matrix for generalized linear systems.
This model reduction algorithm requires only a single low-rank (HAPOD) decomposition of the cross Gramian,
and provides an \linebreak \emph{a-priori} error indicator.
Overall, the cross-Gramian-based dominant subspaces technique is a system-theoretic model reduction method with a simple formulation,
efficient computation, conditional stability preservation and error quantification.
The error indicator for the cross-Gramian-based dominant subspace model reduction could be enhanced,
for example by fitting the known singular values exponentially and incorporate such an empirical decay rate.
The applicability of this method to control-affine nonlinear systems will be subject of future work,
which is in principal possible due to the utilized empirical Gramian computation leading to \emph{empirical dominant subspaces}.

\section*{Code Availability Section}
The source code of the presented numerical examples can be obtained from:
\begin{center}
\url{http://runmycode.org/companion/view/3270}
\end{center}
and is authored by: \textsc{Christian Himpe}.

\section*{Acknowledgement}
Supported by the German Federal Ministry for Economic Affairs and
Energy (BMWi), in the joint project: ``MathEnergy -- Mathematical
Key Technologies for Evolving Energy Grids'', sub-project:
Model Order Reduction (Grant number: 0324019\textbf{B}).


\vskip1ex

This work is dedicated to the late Thilo Penzl,
who wrote the preprint version of \cite{morPen06} twenty years (at this writing) ago, in 1999,
and, moreover,  2019 marks the year of his 20\textsuperscript{th} death anniversary.
Thilo Penzl died December 17\textsuperscript{th}, 1999, but his work and ideas inspire researchers in model reduction and matrix equations to date.

\pagebreak

\bibliographystyle{plainurl}
\bibliography{mor,csc,software} 

\end{document}